\documentclass[11pt,a4paper]{article}
\usepackage{bm}
\usepackage{epsfig}
\usepackage{amsmath}
\usepackage{amssymb}
\usepackage{amsthm}
\usepackage{epsfig}
\usepackage{verbatim}
\usepackage{url}
\usepackage{color}
\usepackage{hyperref}
\include{pspicture}

\newtheorem{theorem}{Theorem}[section]
\newtheorem{lemma}{Lemma}[section]
\newtheorem{corollary}{Corollary}[section]

\newtheorem{remark}{Remark}[section]

\newtheorem{proposition}{Proposition}[section]

\usepackage{mathrsfs}
\title{Joint Hitting-Time Densities for Finite State Markov Processes}

\author{
Tomasz R. Bielecki$^{1,}$\thanks{The research of T.R. Bielecki was supported
 by NSF Grant DMS--0604789 and NSF Grant DMS--0908099.}~, \,
Monique Jeanblanc$^{2,}$\thanks{The research of M. Jeanblanc
benefited from the support of the `Chaire March\'es en mutation',
F\'ed\'eration Bancaire Fran\c caise.} ,  
Ali Devin Sezer $^{2,3}$\thanks{The research of Ali Devin Sezer is supported by Rbuce-up
Marie Curie European Project \url{http://www.rbuceup.eu} }
\\\\
$^1$ Department of Applied Mathematics \\
Illinois Institute of Technology \\ Chicago, IL 60616, USA
\\\\
$^2$ Laboratoire Analyse et Probabilit\'es\\Universit\'e d'\'Evry Val d'Essonne  \\
91025 \'Evry Cedex, France \\  \vspace{6pt}
\\
$^3$ Institute of Applied Mathematics,\\ Middle East Technical University,\\ Ankara, Turkey}
\begin{document}
\maketitle
\begin{abstract}
For a finite state Markov process and a finite collection
$\{ \Gamma_k, k \in K \}$ of subsets of its state space,
let $\tau_k$ be the first time the process visits the set $\Gamma_k$.
We derive explicit/recursive formulas for the joint density and tail probabilities
of
the stopping times
$\{ \tau_k, k \in K\}$. The formulas are natural
generalizations of those associated with the jump times of a simple
Poisson process. We give a numerical example and indicate the relevance of our
results to credit risk modeling.
\end{abstract}

\section{Introduction}
One of the basic random variables associated with a Markov process $X$ is its first hitting
time to a given subset of its state space. In the present work we will confine ourselves to finite state
Markov processes. If $X$ has an absorbing state
and all of the states can communicate with it, the distribution of the first hitting time 
to the absorbing state is said to be a {\em phase-type distribution.} Phase-type distributions
can model a wide range of phenomena in reliability theory, communications systems, 
in insurance and finance and go back to Erlang \cite{erlang1917solution}. The literature on these
distributions is immense, see, e.g., 
\cite{david1987least, johnson1989matching, neuts1981matrix,
syski1992passage,
asmussen2003applied, asmussen2010ruin}. 
To the best of our knowledge, \cite{assaf1984multivariate} 
introduced higher dimensional
versions of phasetype distributions. 
Their setup, for the two dimensional case is as follows: 
take two proper subsets $\Gamma_1$ and $\Gamma_2$ of the state space, and assume that with probability
$1$ the process enters their intersection; let $\tau_k$ be the first time the process
enters $\Gamma_k$. The joint distribution of $(\tau_1,\tau_2)$ is a two dimensional phase type distribution.
Higher dimensional versions are defined similarly for a finite collection of subsets $\{\Gamma_k, k \in K\}$
of the state space.

Denote the number of elements in $K$ by $|K|$.
Multidimensional phase type distributions can put nonzero mass on lower dimensional 
subsets of ${\mathbb R}_+^{|K|}$ and the density of the distribution when restricted to these subsets is
called the ``singular part of the distribution.''
To the best of our knowledge, the only result available in the current 
literature 
giving a complete characterization of any multidimensional phase type 
density is the case of two dimensions treated
in \cite{assaf1984multivariate}. 
The same work derives a density formula for the nonsingular part of a phasetype distribution of arbitrary 
finite dimension, which is proved to be correct in \cite{herbertsson2011modelling}.
The main contribution of the present paper is 
Theorem \ref{t:densityt1tk}, which gives an explicit
formula for the joint density (over appropriate subsets of ${\mathbb R}_+^{|K|}$) of the random
vector $\tau\doteq (\tau_k, k\in K )$
covering all possible singular and nonsingular parts.
It turns out that it is simpler
to work with no assumptions on whether $\{\Gamma_k, k\in K\}$ 
are absorbing or not. 
Thus, Theorem \ref{t:densityt1tk} gives the
joint density of a collection of first hitting times for any finite state
 Markov process $X$. The density of phasetype
densities follows as a special case (see Proposition \ref{p:eltwhenabsorbing} for the one dimensional case
and Proposition \ref{p:densforabs} for the general case).

The primary difficulty in the prior literature in identifying the singular part of the density
seems to have stemmed from the method of derivation, which is:
find expressions for tail probabilites and differentiate them. 
As will be seen in subsection \ref{ss:tailprob}
tail probabilities turn out to be more complicated objects than densities. 
Thus, we follow the opposite 
route and compute first the density directly using the following idea: for each $t \in {\mathbb R}_+^{|K|}$, $\tau= t$ is the limit of
a specific and simple set of trajectories 
of the Markov process whose (vanishing) 
probability can be written in terms of 
the exponentials
of submatrices of the intensity matrix.
Subsection \ref{ss:tau1} explains the idea in its simplest form in the derivation of the density
of a single $\tau_k$, given as Theorem \ref{t:densityt1}. 
The same idea extends to multiple 
hitting time in Subsection \ref{ss:multi} 
and the multidimensional density is given as Theorem \ref{t:densityt1tk}.
Subsection \ref{ss:tailprob} 
derives the tail probabilities of $\tau$ using the same idea; the result is 
given as Theorem \ref{t:tailprobs}. 
The formulas for tail probabilities are 
more complex and are best expressed recursively.
We provide a second formula \eqref{e:recfortailprobsalt}
which explicitly states some of the integration that is hidden in the 
completely recursive \eqref{e:recfortailprobs}.

Let 
$\{ {\mathscr F}_u, u \in {\mathbb R}_+ \}$ be the filtration generated by $X$.
The Markov property of $X$ implies that the conditional
density of $\tau$ given ${\mathscr F}_u$ directly follows
from the density formula \eqref{e:densityt1tk}, which we note as
Proposition \ref{p:conditional}. In Section \ref{s:absorbing}  we derive
alternative expressions for the density
and the tail probability formulas for absorbing $\{\Gamma_k\}$ and indicate
the connections between our results and the prior literature.
Section \ref{s:numerical} gives a numerical example. 
We discuss potential applications of our results to credit risk modeling
and point out several directions of future research in the conclusion.

\section{Definitions}
Let 
$\Omega_0$
be a finite set and $X$ a $\Omega_0$-valued continuous time process 
defined over a measurable pair
$(\Omega, {\mathscr F})$ equipped with a family of measures $P_i$, $i \in \Omega_0$,
such that $P_i(X_0 = i) = 1$. Under each $P_i$,
 $X$ is assumed Markov with intensity matrix $\lambda$.
Denote by $P$ 
the collection of measures $\{ P_i, i \in \Omega_0\}$ 
written as a column.  If $\alpha$ is a probability
measure on $\Omega_0$ (written
as a row), we will
denote by $P_\alpha$ the measure $\alpha P =\sum_{i\in \Omega_0} \alpha(i) P_i$ on $(\Omega,{\mathscr F})$. 
It follows from these definitions that under
$P_\alpha$ the initial distribution of $X$ is $\alpha$, i.e., 
$P_\alpha(X_0 = i) =\alpha(i).$
The total jump rate of the process when in state $i$
is
 $-\lambda(i,i) = \sum_{j\neq i } \lambda(i,j)$.
For a finite collection
$\{ \Gamma_k \subset \Omega_0, k \in K \}$
of subsets of $\Omega_0$
define
$
\tau_k \doteq \inf \{u \in (0,\infty) : X_u \in \Gamma_k \}.$
The index set $K$ can be any finite set, but we will
always take it to be a finite subset of the integers.
In the next section we derive formulas for
the (conditional) joint density and tail probabilities of 
the stopping times $\{ \tau_k, k\in K \}.$
To ease notation, unless otherwise noted,
we will assume throughout that 
 $\Omega_0 - \cup_{k \in K} \Gamma_k$ is not empty and that the initial distribution
$\alpha$ puts its full mass on this set,
see Remark
\ref{rr:alpha}  and subsection \ref{ss:positivemasongamma}
for comments on how one removes this assumption. 

For a set $a$, $a^c$ will mean its complement and if it is finite $|a|$ 
will mean the number of elements in it.
For two subsets $a,b \subset \Omega_0$ define  $\bm{\lambda}(a,b) =\eta$
as 
\begin{equation}\label{e:deflambda}
 \eta(i,j) \doteq \begin{cases}
\lambda(i,j) \text{ if } i\in a, j \in b, \\
 0, \text{ otherwise.}
\end{cases}
\end{equation}
For $a \subset \Omega_0$, we will write $\bm{\lambda}(a)$ for $\bm{\lambda}(a,a).$
We note $\lambda = \bm{\lambda}(\Omega_0)$. 

Throughout we will need to refer to zero matrices and vectors of various
dimensions, we will write all as $0$; the dimension will always be clear
from the context.

For $a \subset \Omega_0$, 
take the identity matrix $I \in {\mathbb R}^{|\Omega_0| \times |\Omega_0|}$ and
replace its rows whose indices appear in $a^c$ with the $0$ vector 
and call the resulting
matrix $I_a$, e.g., $I_{\Omega_0}$ is $I$ itself and $I_\emptyset$ is
the zero matrix. 
The matrix
$I_a$ has the following action on matrices and vectors:
\begin{lemma}\label{l:Iaaction}
Let $n$ be a positive integer.
For any $M \in {\mathbb R}^{|\Omega_0| \times n}$, $I_a M$ is the
same as $M$ except that its rows whose indices are in 
$a^c$ are replaced by $0$ (a zero row vector of dimension $n$), i.e.,
if $r_i$ is the $i^{th}$ row of $M$ then the $i^{th}$ row of 
$I_a M$ is $r_i$ 
if $i \in a$ and $0$ otherwise.
\end{lemma}
The proof follows from the definitions and is omitted.
Right multiplication by $I_a$ acts on the columns, i.e.,
$M I_a$ is the same as
$M$ except now that the columns with indices in $a^c$ are set to zero.
As an operator on
 $|\Omega_0|$ dimensional vectors, $I_a$ replaces with $0$ 
the coordinates of the vector whose indices are in $a^c$.

It follows from the definition \eqref{e:deflambda} 
of $\bm{\lambda}$ and 
Lemma \ref{l:Iaaction}
that
\begin{equation}\label{e:lambdamatrix}
\bm{\lambda}(a,b) = I_{a} \lambda I_{b}.
\end{equation}
The operation of setting some of the columns of the
identity matrix to zero commutes with set operations, i.e., one has
\begin{equation}\label{e:commutation}
I_{a \cap b} = I_a I_b, ~~ I_{a \cup b} = I_a + I_b - I_a I_b, ~~I_{a^c} = I - I_a.
\end{equation}
Using this and Lemma \ref{l:Iaaction} one can write
any formula involving $\bm{\lambda}$ in a number of ways. For example,
$\bm{\lambda}(a^c,a)$ can be written as
$I_{a^c} \lambda I_a = (I - I_{a}) \lambda I_{a} = \lambda I_a - 
I_a \lambda I_a$ or 
$\bm{\lambda}(a, b\cap c)$ as 
$I_a \lambda I_{b\cap c} = I_a \lambda I_b I_c = I_a \lambda I_c I_b.$

\subsection{Restriction and extension of vectors and $\tau$ as a random function}
For any nonempty finite set $a$ let ${\mathbb R}^a$ be the set of functions
from $a$ to ${\mathbb R}$. ${\mathbb R}^{a}$ is the same
as ${\mathbb R}^{|a|}$, except for the way we index the components
of their elements. For two sets $a \subset b$ and $y \in {\mathbb R}^b$
denote 
$y$'s restriction to $a$ by $y|_a \in {\mathbb R}^a$:
\begin{equation}\label{e:restriction}
y|_a(i) \doteq y(i) \text{ for } i \in a.
\end{equation}
The same notation continues to make sense for $a$ of the form $b \times c$,
and therefore can be used to write submatrices of a matrix. Thus, for
$M \in {\mathbb R}^{\Omega_0 \times \Omega_0}$ and nonempty
$b,c \subset \Omega_0$
\begin{equation}\label{e:submatrix}
M|_{b\times c}
\end{equation}
will mean
the submatrix of $M$ consisting of its components $M(i,j)$ with 
$(i,j) \in b \times c.$ For $b=c$ we will write $M|_b.$

For $ x \in {\mathbb R}^a$
denote by $x|^b \in {\mathbb R}^b $ the following extension of $x$ to $b$:
\begin{equation}\label{e:extension}
x|^b(i) = \begin{cases}
		x(i) \text{ for }  i \in a, \\
		0, \text{ otherwise}.
\	\end{cases}
\end{equation}

The random vector $\tau = (\tau_k, k\in K)$ can also
be thought of as a random function on $K$, and we will often do so.
Thus for $ A \subset K$, we may write $\tau|_A$ to denote
$(\tau_k, k \in A)$. 
The advantage of the notation $\tau|_A$
is that we are able to index its components with elements of $A$ rather than
with the integers $\{1,2,3,...,|A|\}$; this proves useful when stating the recursive formulas and proofs below.
 
\subsection{Subpartitions of ${K}$}\label{ss:subpartitions}
The key aspect of the distribution of $\tau$, already referred to
in the introduction, is that it may put nonzero mass on lower dimensional 
subsets of ${\mathbb R}^{|K|}_+.$
This happens, for example,
when $X$ can hit $\cap_{k \in  A} \Gamma_k$ 
before $\cup_{k\in A} \Gamma_k - \cap_{k \in A} \Gamma_k$
with positive probability for some $A \subset K$ with $|A|>1$.
As this example suggests, one can divide ${\mathbb R}^{|K|}_+$ into a number of regions and associate with each an intersection of
events of the form ``$X$ hits $a$ before $b$''
for appropriate subsets of $a,b \subset \Omega_0$.
To write down the various regions and the corresponding events
 we will use  subpartitions of ${K}$,
which we introduce now.

Recall that $K$ is the set of indices of the stopping times 
$\{\tau_k\}$ or
equivalently the sets $\{\Gamma_k\}$.
We call an ordered sequence of disjoint nonempty subsets of ${K}$ a {\em subpartition} of ${K}$. If the union of all
elements of a subpartition is ${K}$ then we call it a partition.
For example,
$( \{1,2\}, \{3\}, \{4\})$  
 [$(\{1,2\},\{4\})$]
is a [sub]partition of $\{1,2,3,4\}$.
Denote by $|s|$ the number of components in the subpartition
and 
by $s(n)$ its $n^{th}$ component, $n \in \{1,2,3,...,|s|\}.$
In which order the sets appear in the partition matters. 
For example,
$( \{3\},  \{4\}, \{1,2\})$ is different from the previous partition.
In the combinatorics literature this is often called
an ``ordered partition,'' see ,e.g., \cite{stanley2011enumerative}. 
Only ordered partitions appear in the present
work and therefore to be brief we always assume every 
subpartition to have a definite order
and drop the adjective ``ordered.'' 
With a slight abuse of notation we will write
$s(n_1,n_2)$ to denote the ${n_2}^{nd}$ element of the ${n_1}^{st}$ set
in the partition. 

Two subpartitions $s_1$ and $s_2$ are said to be disjoint if $\cup_{n} s_1(n)$ and
$\cup_n s_2(n)$ are disjoint subsets of ${K}$.
For a given disjoint pair of subpartitions $s_1$, $s_2$
let $s_1 \cup s_2$
be their concatenation, for example $( \{1,2\},\{3\}) \cup (\{4,6\} ) =
( \{1,2\},\{3\},\{4,6\}).$ 

For a subpartition $s$ let $Ls$ be its left shift, i.e.,
 $L( s(1),s(2),...,s(|s|))$ $= (s(2),s(3),...,s(|s|)).$
Let
$L^m$ denote left shift $m$ times. Similarly
for $ t \in {\mathbb R}^n$, $ n > 1$ let $Lt \in {\mathbb R}^{n-1}$ 
be its left shift.
For $t \in {\mathbb R}^n$ and $r \in {\mathbb R}$ let $t-r$ denote
$(t_1-r,t_2 -r,..., t_n - r).$

Given a subpartition $s$ and an index $0  < n \le |s|$,
let
$s - s(n)$ be
the subpartition which is the same as $s$ but without
$s(n)$, e.g., $( \{1,2\}, \{3\}, \{4,7\} )$ $- \{3\} = (\{1,2\},\{4,7\}).$ 
Given a subpartition $s$ and a nonempty $A \subset {K} - \cup_{n=1}^{|s|} s(n)$
let $s + A$ denote the subpartition
that has all the sets is $s$ and $A$, e.g., $( \{1,2\}, \{3\}) + \{4,7\}  = ( \{1,2\}, \{3\}, \{4,7\}).$ 

Define
\[
S(s) \doteq  \bigcup_{n=1}^{|s|} \bigcup_{k \in s(n) } \Gamma_k.
\]
For a {\em partition} $s$ define
\begin{align*}
{\mathbb R}_+^K \supset R_s \doteq &
\left( \bigcap_{n=1}^{|s|} 
\bigcap_{k_1, k_2 \in s(n) } 
\left\{ t_{ k_1 } = t_{k_2} \right\}\right)
    \cap \left\{ t_{s(1,1)}   < t_{ s(2,1)} < \cdots < t_{s(|s|,1)}\right\}.
\end{align*}
For example, for $s=(\{1,4\}, \{2\},\{3,5,6\})$ 
\[
R_s= \{ t: t_1 = t_4 < t_2 < t_3= t_5 = t_6 \}.
\]
Let ${\mathcal S}$ be the set of all partitions of ${K}$. 
The sets
$R_s, s \in {\mathcal S}$,  are disjoint and their union is ${\mathbb R}_+^K$. 
It turns out that for each $s \in {\mathcal S}$, the distribution of $\tau$ restricted to $R_s$ is absolutely continuous
with respect to the $|s|$ dimensional Lebesgue measure on $R_s$. Our main result, given as Theorem \ref{t:densityt1tk} below,
 is a formula for this density. 

\section{The density of first hitting times}
We start by deriving the density of a single hitting time over sets of sample paths
that avoid a given subset of the state space until the hitting occurs.
\subsection{Density of one hitting time}\label{ss:tau1}
For any set $d\subset \Omega_0$ and $u \in {\mathbb R}_+$ define
$p_{\alpha,d}^u(j)  \doteq P_\alpha( X_u = j, X_v \notin d, v \le u )$
and $ p^u_{d}(i,j) \doteq P_i( X_u = j, X_v \notin d, v \le u ).$
In addition set $p^u(i,j) \doteq p^u_\emptyset(i,j) =  P_i( X_u = j)$.
The distribution $p_{\alpha,d}^u$ is a row vector and 
$p^u_d$ and $p^u$ are $|\Omega_0| \times |\Omega_0|$ matrices.
Conditioning on the initial state
implies
$p_{\alpha,d}^u  = \alpha p^u_d.$
It follows from the definition of $X$, $\lambda$ and $p^h$ that
\begin{equation}\label{e:intensities}
\lim_{h\rightarrow 0} p^h(i,j)/h = \lambda(i,j),
\end{equation}
for $(i,j) \in d^c \times d^c.$ 
\begin{lemma}\label{l:avoidG1}
Let $\alpha$ be an initial distribution on $\Omega_0$ with $\alpha|_d = 0$. Then
\begin{equation}\label{e:toprovel1}
p_{\alpha,d}^u = \alpha e^{u \bm{\lambda}(d^c) }.
\end{equation}
\end{lemma}
\begin{proof}
We only need to modify slightly the proof of \cite[Theorem 3.4, page 48]{asmussen2003applied}. 
The steps are: 1) write down a linear ordinary differential
equation (ODE) that the matrix valued function $ u\rightarrow p^u_{d}|_{d^c}$,
$u \in {\mathbb R}_+$, satisfies, 2) the 
basic theory of ODEs will tell us
that the unique solution is $u\rightarrow e^{u\bm{\lambda}(d^c)}|_{d^c}$.

Let $\nu_1$ be the first jump time of $X$; for $X_0 = i\in d^c$, 
$\nu_1$
is exponentially distributed with rate $-\lambda(i,i)>0$.
Conditioning on $\nu_1$ gives
{\small
\begin{equation}\label{e:ode}
p_{d}^u(i,j) = 
P_i(\nu_1 > u) I(i,j) + \int_0^u \lambda(i,i) 
e^{\lambda(i,i) v} \left( \sum_{l \in d - \{i\}}
\frac{\lambda(i,l)}{\lambda(i,i)} p_{d}^{u-v}(l,j) \right) dv
\end{equation}
}
for $(i,j) \in d^c \times d^c$.
In comparison with the aforementioned proof
we have only changed the index set of the last sum to ensure that only paths
that keep away from $d$ are included. The unique solution of
\eqref{e:ode} equals
$p_d^u|_{d^c} = e^{ u\lambda|_{d^c}} = e^{u\bm{\lambda}(d^c)}|_{d^c}.$
The equality
\eqref{e:toprovel1} follows from this and $\alpha|_d =0.$
\end{proof}

\begin{remark}{\em
Probabilities that concern sample paths that stay away from a given
set are called ``taboo probabilities'' in \cite[Section 1.2]{syski1992passage};
\cite[Equation (F), page 28]{syski1992passage} is equivalent to \eqref{e:ode}.
}
\end{remark}

The next result (written in a slightly different form) 
is well known, see, e.g.,
\cite{neuts1975probability,assaf1984multivariate}. We record it
as a corollary here and will use it in 
subsection \ref{ss:linkstoearlier}
where we indicate the connections of our results 
to prior literature.
Let ${\bf 1}$ be the $|\Omega_0|$ dimensional column vector with all components
equal to $1$.
\begin{corollary}
For $\tau_d \doteq \inf \{ u: X_u \in d \}$, and an initial 
distribution with $\alpha|_d = 0$
\begin{equation}\label{e:tail1}
P_\alpha( \tau_d > u) = \alpha e^{u\bm{\lambda}(d^c) } {\bf 1}.
\end{equation}
\end{corollary}
\begin{proof}
\begin{align*}
P_\alpha( \tau_d > u) = \sum_{j \in d^c} 
P_\alpha(X_u = j, X_v \notin d, v \le u) = \alpha e^{u\bm{\lambda}(d^c) } {\bf 1},
\end{align*}
where the last equality is implied by \eqref{e:toprovel1}.
\end{proof}

\begin{remark}\label{rr:alpha}
{\em
One must modify \eqref{e:tail1} to
\[
P_\alpha( \tau_d >  u) = 
\alpha 
I_{d^c} 
e^{u\bm{\lambda}(d^c) } {\bf 1},~~~
P_\alpha(\tau_d = 0)= \alpha I_d {\bf 1}
\]
if one does not assume $\alpha|_d = 0$.
}
\end{remark}

\begin{theorem}\label{t:densityt1}
Let $a,b \subset \Omega_0$, $a\cap b = \emptyset$ be given.
Define $\tau_a \doteq\inf \{u: X_u \in a \}$ and set $d = a \cup b.$
Then
\begin{equation}\label{e:dist0}
\frac{d}{du} \left[ P_\alpha(\tau_a \in (0,u], X_v \notin b, v \le \tau_a ) \right]= 
\alpha e^{u\bm{\lambda}(d^c)} \bm{\lambda}(d^c,a) {\bf 1},
\end{equation}
where $\alpha$ is the initial distribution of $X$ with
 $\alpha|_{d} = 0$.
\end{theorem}

The idea behind \eqref{e:dist0} and its proof is this:
for $\tau_a = u$ with $X$ staying out of $b$ until time $u$,  $X$ has to stay in the set $d^c$
until time $u$ and jump exactly at that time into $a$.

\begin{proof}[Proof of Theorem \ref{t:densityt1}]
The definition of the exponential distribution implies that that $X$
jumps more than once in during the time interval $[u,u+h]$ 
has probability $O(h^2)$.  This,
\eqref{e:intensities} and the Markov property of $X$ (invoked at time $u$)
give
\begin{equation}\label{e:basic}
P_i( \tau_a \in (u,u+h), X_v \notin d, v \le u ) = \left( \sum_{j \in a} 
\sum_{l \in d^c } p_{d}^u(i,l) ~\lambda(l,j) \right) h + o(h).
\end{equation}
By the previous lemma $p_{d}^u(i,j)$ equals
exactly the $(i,j)^{th}$ component of  $e^{u\bm{\lambda}(d^c)}$. These imply  
\eqref{e:dist0}.
\end{proof}

The ideas in the previous proof also give
\begin{proposition}\label{p:distXtaub}
Let $a,b \subset \Omega_0$,  $a\cap b = \emptyset$, $a$ nonempty  be given.
Define $\tau_a \doteq \{u: X_u \in a \}$ and $d = a\cup b.$
Let $\alpha$ is an initial distribution on $\Omega_0$ with
$\alpha|_d = 0.$
Set
$\alpha_1 \doteq \alpha e^{\tau_a \bm{\lambda}(d^c)} \bm{\lambda}(d^c,a)$
and ${\mathcal V} \doteq \{X_v \notin b,v \le \tau_a \}$. Then
\[
P_\alpha 
( X_{\tau_a} = j | (\tau_a, 1_{\mathcal V} ) ) = 
\alpha_1(j) / \alpha_1 {\bf 1} \text{ on }  {\mathcal V},
\]
where $1_{\mathcal V}$ is the indicator function of the event
${\mathcal V}.$
\end{proposition}
${\mathcal V}$ is the event that $X$ does not visit the set $b$
before time $\tau_a$.
\begin{proof}
The arguments that led to \eqref{e:basic} 
in the proof of Theorem \ref{t:densityt1} also give
\begin{align*}
P_i( X_{\tau_a }= j, \tau_a  \in (u,u+h),~ &X_v \notin b, v \le u ) \\
&~~~= \left( 
\sum_{l \in d^c } p_{d}^u(i,l) ~\lambda(l,j) \right) h + o(h).
\end{align*}
The rest follows from the
definition of the conditional expectation.
\end{proof}

Set $b=\emptyset$ in Theorem \ref{t:densityt1} to get the density
of $\tau_a.$ The formula \eqref{e:dist0} generalizes 
the exponential density: if $\tau'$ is exponentially distributed
with rate $\lambda' \in (0,\infty)$ it has density $e^{\lambda' t} \lambda'$.

\subsection{The multidimensional density}\label{ss:multi}
One can extend \eqref{e:dist0} to a representation of the distribution of 
$\tau$ using
the subpartition
notation of subsection \ref{ss:subpartitions}. 
For a partition $s$ of $K$,
$n \in \{1,2,...,|s|\}$ and
$t\in R_s \subset {\mathbb R}_+^{K}$ 
define
\begin{equation}\label{e:bartWT}
\bar{t}_n \doteq t_{s(n,1)},~~~ \bar{t}_0 \doteq 0, ~~
W_n \doteq [S(L^{n-1}s)]^c,
 T_n \doteq \left[\bigcap_{ k \in s(n) }\Gamma_k\right] \cap W_{n+1},
\end{equation}
where $W$ stands for ``waiting''
and $T$ for ``target.'' The key idea of the density formula and its proof is the $|s|$ step version
of the one in Theorem \ref{t:densityt1}:
in order for $\tau=t \in {\mathbb R}_+^K$, 
$X$ has to stay in the set $W_1$ until time $\bar{t}_1$ and jump exactly at
that time into $T_1 \subset W_2$; then stay in the set $W_2$ until time $\bar{t}_2$ and jump
exactly then into $T_2$ and so on until all of the pairs $(W_n,T_n)$, $n\le |s|$,
are exhausted.

Although not explicitly stated,
all of the definitions so far depend on the collection $\{\Gamma_k, k \in K\}$.
We will
express this dependence explicitly in the following theorem by including
the index set $K$ as a variable of the density function $f$. This will
be useful in its recursive proof, in the next subsection
where we comment on the case when
$\alpha$ is an arbitrary initial distribution
and in Proposition \ref{p:conditional}
where we give the conditional density of $\tau$ given ${\mathscr F}_u$, 
$u > 0$.
For a sequence $M_1, M_2,..., M_n$ of square matrices of the same size
$\prod_{m=1}^n M_m$
will mean
$M_1 M_2 \cdots M_n$.
\begin{theorem}\label{t:densityt1tk}
For any partition $s \in {\mathcal S}$ of $K$,
the distribution of $\tau$ on the set $R_s$ has density
\begin{equation}\label{e:densityt1tk}
f(\alpha, t,K) \doteq \alpha \left( \prod_{n=1}^{|s|} e^{\bm{\lambda}(W_n)(\bar{t}_n -\bar{t}_{n-1})}  
\bm{\lambda}(W_n,T_n) \right) {\bf 1},
\end{equation}
$t \in R_s \subset{\mathbb R}^K_+$, 
with respect to the $|s|$ dimensional Lebesgue measure on $R_s.$
\end{theorem}
In the proof we will use
\begin{lemma}\label{l:computecondexp}
Let ${\mathscr S}_1$ and ${\mathscr S}_2$ 
be two measurable spaces and
$g: {\mathscr S}_1 \times {\mathscr S}_2 \rightarrow {\mathbb R}$ a 
bounded measurable function. 
Let $Y_i$ 
be an ${\mathscr S}_i$ valued random variable
on a probability space $(\bar{\Omega}, \bar{\mathscr F}, \bar{\mathbb P})$. 
Let ${\mathscr G}$
be a sub $\sigma$-algebra of 
$\bar{\mathscr F}$ and suppose 1)  $Y_1$ is ${\mathscr G}$
measurable and 2)  under $\bar{\mathbb P}$, $Y_2$ 
has a regular conditional distribution given ${\mathscr G}.$ 
For $y_1 \in {\mathscr S}_1$ 
define $h(y_1) \doteq \bar{\mathbb E}[ g(y_1, Y_2) | {\mathscr G}].$
Then
\[
\bar{\mathbb E}[g(Y_1,Y_2) |{\mathscr G}] = h(Y_1).
\]
\end{lemma}
The value $h(y_1)$ in the previous lemma is defined via a conditional expectation and 
therefore it depends on $\bar{\omega} \in \bar{\Omega}$.
The proof of Lemma \ref{l:computecondexp} 
follows from the definition of regular conditional distributions, see, for example,
\cite[Section 5.1.3, page 197]{durrett2010probability}. 
To invoke Lemma \ref{l:computecondexp} we need the
existence of the regular conditional distribution of $Y_2$;
$Y_2$ in the proof below will take values in
a finite dimensional Euclidean space (a complete separable metric
space) and therefore will have a regular conditional distribution, 
for further details
we refer the reader to, e.g., \cite[Theorem 2.1.15]{durrett2010probability} and 
\cite[Theorem 5.1.9]{durrett2010probability}.
\begin{proof}
The proof will use induction on $|K|.$
For $|K|=1$ \eqref{e:densityt1tk} (with $b= \emptyset$)
and \eqref{e:dist0} are the same. Suppose
that \eqref{e:densityt1tk} holds for all $K$ with $|K| \le \kappa -1$;
we will now argue that then it must also hold for $|K|=\kappa.$ 
Fix a partition $s$ of $K$. We would like to show
that $\tau$ restricted to $R_s$ has the density \eqref{e:densityt1tk}.

For any continuous 
$g:{\mathbb R}^{K} \rightarrow {\mathbb R}$ with compact
support, we would like to show
\begin{equation}\label{e:toshowdensitythm}
{\mathbb E}[ 1_{R_s}(\tau) g(\tau) ] = \int_{R_s} g(t) f(\alpha, t,k) d_s t,
\end{equation}
where $d_s t$ denotes the $|s|$ dimensional Lebesgue measure on $R_s$.
Define 
$\tau' \doteq \bigwedge_{k \in K} \tau_k$; $\tau'$ is the first time
$X$ enters $\cup_{k\in K} \Gamma_k.$
In the rest of the proof we will proceed as if $ P_\alpha(\tau' < \infty)=1$;
the treatment of the possibility $P_\alpha(\tau' = \infty)>0$ needs no
new ideas and the following argument can
be extended to handle it by adding several case by case comments.

If $\tau \in R_s$ holds then
1)
$X_{\tau'} \in T_1$ and 2) $X_t \in W_1$ for $ t \le \tau'$;
1) and 2) also imply $\tau' = \tau_{s(1,1)}.$
Therefore,
\begin{equation}\label{e:RssubsetW1}
\{ \tau \in R_s \} \subset 
{\mathcal W}_1 \doteq 
\{ X_u \in W_1, u \le \tau'\} \cap \{ X_{\tau'} \in
T_1 \}.
\end{equation}
Theorem \ref{t:densityt1} implies that $\bm{\lambda}(W_1,T_1)$ is non zero
if and only if ${\mathcal W}_1$ has nonzero probability. Thus
if
$\bm{\lambda}(W_1,T_1)$ is zero then
$P_\alpha(\tau \in R_s) =0$
and indeed $f(\alpha,t,K)=0$ is the density of $\tau$ over $R_s.$
From here on we will treat the case
when $\bm{\lambda}(W_1,T_1)$ is nonzero.

Define $\hat{X}_u \doteq X_{u + \tau'}$ and for $k \in S(Ls)$
\[
\hat{\tau}_k \doteq \inf \{u : \hat{X}_u \in \Gamma_k \};
\]
one obtains $\hat{X}$ from $X$ by shifting time for the latter
left by $\tau'$, i.e., once time hits $\tau'$ reset it to $0$
and call the future path of the process $\hat{X}$.
The Markov property of $X$ implies that
$\hat{X}$ is a Markov process with intensity matrix $\lambda$
and initial point $\hat{X}_0= X_{\tau'}$.
The relation
\eqref{e:RssubsetW1} implies 
\begin{equation}\label{e:rectau0}
\hat{\tau} = \tau|_{Ls} - \tau'
\end{equation}
on the set $\{\tau \in R_s\}.$ Finally, the last display, 
the definition of $\tau$ and that of ${\mathcal W}_1$ imply
\begin{equation}\label{e:Rsrec}
\{ \tau \in R_s \} = {\mathcal W}_1 \cap \{ \hat{\tau} \in R_{Ls} \}.
\end{equation}
In words this display says: for $\tau$ to be partitioned
according to $s$,  among all $\{\Gamma_k\}$, $X$ must visit $\cap_{k\in s(1)} \Gamma_k$ first
and after this visit the rest of the hitting times
must be partitioned according to $Ls$.

Denote by ${\bf 1}'$ the function that maps all elements of $K$ to $1$.
Define $\hat{g}: {\mathbb R} \times {\mathbb R}^{S(Ls)} \rightarrow {\mathbb R}$
as
\[
\hat{g}(t', \hat{t}) \doteq g\left(
t' {\bf 1}' + \hat{t} |^{S(s)} \right),
\]
where
we use the function restriction/ extension notation
 of \eqref{e:restriction}  and \eqref{e:extension}.
Displays
\eqref{e:rectau0} and \eqref{e:Rsrec} imply
\begin{align}
{\mathbb E}[ 1_{R_s}(\tau) g(\tau) ] &= \notag
{\mathbb E}[ 1_{{\mathcal W}_1} 1_{R_{Ls}}(\hat{\tau})\hat{g}(\tau', \hat{\tau} ) ].
\intertext{Condition the last expectation on ${\mathscr F}_{\tau'}$:}
&=
\notag
{\mathbb E}[ 
{\mathbb E}[1_{{\mathcal W}_1} 1_{R_{Ls}}(\hat{\tau})\hat{g}(\tau', \hat{\tau} )| 
{\mathscr F}_{\tau'}] ].
\intertext{ ${\mathcal W}_1$ is ${\mathscr F}_{\tau'}$ measurable and gets out
of the inner expectation}
&=
\label{e:densitythmarg1}
{\mathbb E}[ 
1_{{\mathcal W}_1} 
{\mathbb E}[
1_{R_{Ls}}(\hat{\tau})g(\tau', \hat{\tau} )| 
{\mathscr F}_{\tau'}] ].
\end{align}

For $t' \in {\mathbb R}_+$ define
\begin{equation}\label{e:defh}
h(t') \doteq
 {\mathbb E}[ 1_{R_{Ls}}(\hat{\tau}) \hat{g}(t', \hat{\tau} )
| {\mathscr F}_{\tau'} ] =
 {\mathbb E}[ 1_{R_{Ls}}(\hat{\tau}) \hat{g}(t', \hat{\tau} )
| \hat{X}_0 ],
\end{equation}
the last equality is by the strong Markov property of $X$.
Once again, $h(t')$ is 
a conditional expectation and thus it depends on $\omega$.
Lemma \ref{l:computecondexp} implies that the conditional expectation
in \eqref{e:densitythmarg1} equals
$h(\tau')$ ($\hat{\tau}$ is substituted for the $Y_2$ of the lemma).
The random variable $\hat{X}(0)$ takes values in a finite set
and therefore one can compute the
last conditional expectation by conditioning
on each of these values separately. This, that $\hat{X}$ is a Markov
process with intensity matrix $\lambda$ and
the induction assumption imply that on 
$\hat{X}_0 = j$
\begin{equation}\label{e:innercondexp}
h(t')= {\mathbb E}[ 1_{R_{Ls}}(\hat{\tau}) \hat{g}(t', \hat{\tau} )
| \hat{X}_0 =j ]
= \int_{R_{Ls}} f(\delta_j, t, K-s(1) ) g(t', \hat{t}) d_{Ls} t.
\end{equation}
Once we substitute \eqref{e:innercondexp} for the conditional expectation
in \eqref{e:densitythmarg1} we get an expectation involving only three
random variables: $\tau'$, $1_{{\mathcal W}_1}$ 
and $\hat{X}_0= X_{\tau'}.$  Theorem \ref{t:densityt1} implies that
the density of $\tau'$ on the set ${\mathcal W}_1$
is
$\alpha e^{\bm{\lambda}(W_1) \bar{t}_1} \bm{\lambda}(W_1,T_1) {\bf 1}$
and Proposition \ref{p:distXtaub} implies that the 
distribution of $\hat{X}(0)$ conditioned on $\tau' =\hat{t}_1$ and 
$1_{{\mathcal W}_1} = 1$ 
is
\[
\frac{\alpha e^{\bm{\lambda}(W_1) \hat{t}_1 } \bm{\lambda}(W_1,T_1)}{\alpha
e^{\bm{\lambda}(W_1) \hat{t}_1} \bm{\lambda}(W_1,T_1) {\bf 1}.
}
\]
These, the induction hypothesis, \eqref{e:defh} and \eqref{e:innercondexp}
imply that the outer expectation \eqref{e:densitythmarg1}
equals \eqref{e:toshowdensitythm}.
This last assertion finishes the proof 
of the induction step and hence the theorem.
\end{proof}
In what follows, to ease exposition,
we will sometimes refer to $f$ as the ``density'' of $\tau$
without explicitly mentioning the reference measures $d_s$, $s \in {\mathcal S}$.

\begin{remark}{\em
If any of the matrices in the product \eqref{e:densityt1tk}
equals the zero matrix  then $f$ will be $0$. 
Therefore,
if $\bm{\lambda}(W_n,T_n)=0$ for some $n =1,2,...,|s|$ then
$P_\alpha( \tau \in R_s) = 0.$
By definition  $\bm{\lambda}(W,T)=0$
if $T = \emptyset.$
Thus as a special case we have
$P_\alpha( \tau \in R_s) = 0$ if
$T_n =\emptyset$ for some $n = 1,2,3,...,|s|$.
}
\end{remark}

\begin{remark}{\em
The first $\kappa > 0 $ jump times
of a standard Poisson process with rate $\lambda' \in (0,\infty)$
have the joint density
\[
\prod_{n=1}^{\kappa} e^{\lambda' (t_{n} - t_{n-1})} \lambda',
\]
$0=t_0 < t_1 < t_2 < \cdots < t_\kappa$.
The density
\eqref{e:densityt1tk} is a generalization of this simple formula.
}
\end{remark}

\subsection{When $\alpha$ puts positive mass on $\cup_{k \in K} \Gamma_k$}
\label{ss:positivemasongamma}
If $\alpha$ puts positive
mass on $\gamma \doteq \cup_{k \in K} \Gamma_k$
one best describes the distribution of $\tau$ piecewise as follows.
Define
$\bar{\alpha}' \doteq  1- \sum_{ i \in \gamma } \alpha(i)$
and 
$\alpha' \doteq (\alpha  - \sum_{i \in \gamma } \alpha(i) \delta_i)/
\bar{\alpha}'$
if $\bar{\alpha}' > 0$; $\bar{\alpha}'$ is a real number and $\alpha'$,
when defined, is a distribution.
First consider the case when $\bar{\alpha}' > 0$. The foregoing
definitions imply
\begin{equation}\label{e:inpieces}
P_\alpha( \tau \in U) = \bar{\alpha}' P_{\alpha'}( \tau \in U) + 
\sum_{ i \in \gamma } \alpha(i) P_i ( \tau \in U)
\end{equation}
for any measurable
$U \subset {\mathbb R}_+^K$. By its definition $\alpha'$ puts no mass
on $\gamma = \cup_{k \in K } \Gamma_k$ and therefore 
Theorem \ref{t:densityt1tk}
is applicable and 
$f(\alpha', \cdot , K)$ is the density of the distribution
$P_{\alpha'}(\tau \in \cdot)$.
For the second summand of \eqref{e:inpieces}, it is
enough to compute each $P_i(\tau \in U )$ separately.
Define
$K_i \doteq \{k:  i \in \Gamma_k \}$,
$U_i \doteq \{t: t\in U, t_k = 0, k \in K_i\}$,
$\bar{U}_i \doteq \{ t|_{K_i^c}, t \in U_i \}.$
Now remember that $i \in \gamma$; thus 
if $i \in \Gamma_k$ then $\tau_k = 0$ under $P_i$, and
therefore
$P_i( \tau \in U) = P_i ( \tau \in U_i).$
For $\tau \in U_i$, the stopping times $\tau|_{K_i}$ are all deterministically
$0$. Thus to compute $P_i( \tau \in U_i)$ it suffices to compute
$P_i ( \tau|_{K_i^c}  \in \bar{U}_i).$ But by definition
$i \notin \cup_{ k \in K_i^c} \Gamma_{k}$ and once again Theorem
\ref{t:densityt1tk} is applicable and gives the density of 
$\tau|_{K_i^c}$ under $P_i$ as $f(\delta_i, \cdot, K_i^c).$
If $\bar{\alpha}' = 0$ then 
\[
P_\alpha( \tau \in U) = 
\sum_{ i \in \gamma } \alpha(i) P_i ( \tau \in U)
\]
and the computation of $P_i(\tau \in U)$ goes as above.

\subsection{Tail probabilities of $\tau$}\label{ss:tailprob}
By tail probabilities we mean probabilities of sets of the form
{\small
\begin{equation}\label{e:tailevents}
\bigcap_{n=1}^{|s|} 
\bigcap_{k_1,k_2 \in s(n) } \{ \tau_{k_1} = \tau_{k_2} \} \cap 
\left \{ \tau_{s(n,1)} > t_{n}
\right\} \bigcap_{ n_1 \neq n_2 , n_1,n_2 \le |s| } \{ \tau_{s(n_1,1)} \neq \tau_{s(n_2,1)} \},
\end{equation}
}
where  $s$ is a partition of $K$  and $t \in {\mathbb R}_+^{|s|}$ such that
$t_n < t_{n+1}$, $n = 1,2,3,...,|s|-1.$
Thus this definition of tail events require
that every equality and inequality condition be explicitly specified.
One can write standard tail events in terms of these, 
e.g.,
$\{ \tau_1 > t_1 \} \cap \{\tau_2 > t_2\}$
is the same as the disjoint union
\[\left( \{\tau_1 > t_1 ,\tau_2 > t_2 \} \cap \{\tau_1 \neq \tau_2\}
\right )
\cup \{\tau_1 = \tau_2 > \max(t_1,t_2)\}.
\] Both of these sets are of the form
\eqref{e:tailevents}.
Thus, it is enough to be able to compute
probabilities of the form \eqref{e:tailevents}.
From here on, to keep the notation short, we will assume
that, over tail events, unless explicitly stated with an equality condition, 
all stopping times appearing in them
are strictly unequal to each other (therefore, when writing formulas,
we will omit the last intersection in \eqref{e:tailevents}).

A tail event of the form \eqref{e:tailevents} 
consists of a sequence of constraints of the form
\[
\{ \tau_{s(n,1)}  = \tau_{s(n,2)}= \cdots = \tau_{s(n,|s(n)|)} > t_n \}.
\]
There are two types of subconstraints involved here: that entrance to 
all $\Gamma_{k} $, $k\in s(n)$, happen at the same time and that 
this event occurs after time $t_n$.
Keeping track of all of these constraints as they evolve in time 
requires more notation, which we now introduce.

Take two disjoint subpartitions $s_1$ and $s_2$ of ${K}$ and
an element $t \in {\mathbb R}_+^{|s_1|}$ such that $t_{|s_1|} > t_{|s_1|-1}
> \cdots > t_2 > t_1$; if $|s_1|=0$ by convention set $t=0.$
Generalize the class of tail events to
\begin{align}\label{e:extendedtail}
{\cal T}(s_1,s_2,t) &\doteq
 \Omega \cap
\left(
\bigcap_{n=1}^{|s_1|} ~~\bigcap_{k_1, k_2 \in s_1(n) } \{ \tau_{k_1} = \tau_{k_2} \} 
\cap \left \{ \tau_{s_1(n,1)} > t_{n}  \right\}\right)\cap\notag \\
&~~~~~~~\bigcap_{n=1}^{|s_2|} ~~\bigcap_{k_1,k_2 \in s_2(n) } \{ \tau_{k_1} = \tau_{k_2} \}.
\end{align}
Setting $s_1 = s$ and $s_2 =\emptyset$ reduces \eqref{e:extendedtail} to \eqref{e:tailevents}.
The indices in $s_1$ appear both in
equality constraints  and time constraints
while indices in $s_2$ appear only in equality constraints.
\begin{remark}\label{r:singleelementcomp}
{\em
The definition \eqref{e:extendedtail} implies that
if a component of $s_2$ has only a single element, that component has no influence on ${\cal T}(s_1,s_2,t)$.
For example, ${\cal T}(s_1,(\{1\},\{2,3\}),t)$ is the same as ${\cal T}(s_1,(\{2,3\}), t).$
}
\end{remark}
To express $P_\alpha({\cal T}(s_1,s_2,t))$ we 
will define a collection of functions ${\bf p}_i$,
$i\in \Omega_0$,
of $s_1$, $s_2$ and $t$.
Let ${\bf p}$ be the collection $\{ {\bf p}_i, i \in \Omega_0\}$ written
as a column matrix.
For $s_1 = \emptyset$,
and $i \in \Omega_0$ define ${\bf p}_i$ as
\[
{\bf p}_i(\emptyset, s_2, 0)\doteq  P_i( {\mathcal T}(\emptyset ,s_2,0)).
\]
The definitions of ${\bf p}$
 and ${\mathcal T}$ and Remark \ref{r:singleelementcomp}
imply 
\begin{equation}\label{e:trivials2}
{\bf p}(\emptyset, s_2 , 0) ={\bf 1}
\end{equation} if $s_2$ is empty or it consists of components with single elements.
For a given disjoint pair of subpartitions $s_1$, $s_2$ 
define
\[
T_n(s_1,s_2) \doteq 
\bigcap_{k \in s_2(n)} \Gamma_k - S(s_1 \cup s_2 - s_2(n)),~~
T(s_1,s_2) \doteq \bigcup_{n=1}^{|s_2|} T_n(s_1,s_2).
\]
If $s_1 \neq \emptyset$ define
\begin{align}\label{e:recfortailprobs}
&{\bf p}(s_1,s_2,t) \doteq \\
&~~
\int_0^{t_1} e^{u\bm{\lambda}(W)} \bm{\lambda}(W, T(s_1,s_2))
\left( \sum_{n=1}^{|s_2|} I_{T_n(s_1,s_2)}~~ {\bf p}(s_1, s_2 - s_2(n), t-u)\right) du
\notag \\
&~~~+e^{t_1 \bm{\lambda}(W)} {\bf p}\left( Ls_1 , s_2 + s_1(1), Lt - t(1) 
\right),\notag
\end{align}
where $W = [S(s_1 \cup s_2)]^c.$
If $s_1\neq \emptyset $ and $s_2 = \emptyset$
\eqref{e:recfortailprobs} reduces to
\begin{equation}\label{e:tailsc}
{\bf p}(s_1,\emptyset,t) = e^{\bm{\lambda}(S(s_1)^c )t_1} {\bf p}(L s_1,
( s_1(1) ), Lt - t_1 ).
\end{equation}
We have the following representation of tail probabilities:
\begin{theorem}\label{t:tailprobs}
Suppose $\Omega_0 - S(s_1 \cup s_2)$ is not empty and that
$\alpha$ is an initial distribution on $\Omega_0$ that puts all of its
mass on this set. Then
\[
P_\alpha({\mathcal T}(s_1,s_2, t) ) = \alpha {\bf p}(s_1,s_2,t).
\]
\end{theorem}
The proof is parallel to the proof of Theorem \ref{t:densityt1tk} 
and involves the same ideas and is omitted.

One can write
\eqref{e:densityt1tk} recursively, similar to \eqref{e:recfortailprobs}. The reverse
is not true:
equality constraints, when present, preclude a simple explicit formula for ${\bf p}$ similar to \eqref{e:densityt1tk},
but see subsection \ref{ss:secreptail} for a slightly more explicit representation of ${\bf p}$.

When $s_1$ has no equality constraints and $s_2 =\emptyset$, one can invoke \eqref{e:tailsc} $|s_1|$
times along with Remark \ref{r:singleelementcomp} and \eqref{e:trivials2} and get
\begin{corollary} Let $\alpha$ be as in Theorem \ref{t:tailprobs}.
If $|s_1| > 0$ equals the dimension of $t$ then
\begin{equation}\label{e:tailwhens1simple}
\alpha {\bf p}(s_1,\emptyset,t) = \alpha \left( \prod_{n=1}^{|s|} e^{\bm{\lambda}(W_n)(t_n-t_{n-1})}\right) {\bf 1} 
\end{equation}
where $W_n = [S(L^{n-1}(s_1))]^c$.
\end{corollary}
The formula
\eqref{e:tailwhens1simple} is a generalization of \cite[equation (7)]{assaf1984multivariate} to general finite state Markov
processes. 

If $s_1 = \emptyset$ we have no time constraints
and $P_\alpha({\mathcal T}(\emptyset, s_2,0))$
reduces
to the probability that certain equality and inequality constraints
hold among the stopping times.
This can
be written as the solution of a sequence of linear equations whose
defining matrices are submatrices of the intensity matrix. The details
require
further notation and are left to future work (or to the reader)
except for the
special case of $P_\alpha( \tau_1 = \tau_2)$ which we would like
use in what follows to relate our results to earlier works in the literature.

Define $\nu_0 \doteq 0$,  and for $n > 0$
$\nu_n \doteq \inf \{u > \nu_{n-1}, X_{u} \neq X_{u-} \}.$
The sequence $\{\nu_n\}$ is the jump times of the process $X$.
Define
$\bar{X}_n \doteq X_{\nu_n}.$
$\bar{X}$ is a discrete time Markov chain with state space $\Omega_0$;
it is called the embedded Markov chain of the process $X$.
It follows from \eqref{e:intensities}
that the one step transition matrix of $\bar{X}$  is
\[
\bar{\lambda} \doteq \begin{cases}
-\lambda(i,j)/\lambda(i,i), &\text{ for } i \neq j,\\
	0,&\text{ otherwise.}
\end{cases}
\]
Define $D \in {\mathbb R}^{\Omega_0 \times \Omega_0}$ as
the diagonal matrix 
\[
D(i,j) = 
\begin{cases} -1/\lambda(i,i),&\text{  if } i=j,\\
		0,           & \text{ otherwise.}
\end{cases}
\] 
Left multiplying a matrix by $D$ divides its $i^{th}$ row
by $-\lambda(i,i).$ Therefore, $\bar{\lambda} = 
I + D\lambda.$

Define $\bar{\tau}_k \doteq \inf \{ n: \bar{X}_n \in \Gamma_k \}.$
The event 
$\{\tau_1 = \tau_2\}$ means that $X$ hits the set $\Gamma_1$ 
and $\Gamma_2$ at the same time; 
because this event makes no reference to how
time is measured, it can also be expressed in terms of $\bar{X}$
as $\{\bar{\tau}_1 = \bar{\tau}_2 \}$.

Define the column vector
$q \in {\mathbb R}^{\Omega_0}$, 
$q(i) \doteq P_i( \bar{\tau}_1 = \bar{\tau}_2 ).$ 
Conditioning on the initial position of $X$ implies 
$P_\alpha(\bar{\tau}_1 = \bar{\tau}_2) = \alpha q.$ From here on
we derive a formula for $q$.
Parallel to the arguments
so far, we know that this event happens if and only if $\bar{X}$ hits
$\Gamma_1 \cap \Gamma_2$ before 
$B\doteq (\Gamma_1 - \Gamma_2 ) \cup (\Gamma_2 - \Gamma_1)$. 
Set 
$w \doteq (\Gamma_1 \cup \Gamma_2)^c$. 
$q$ satisfies the boundary conditions 
\begin{equation}\label{e:bcq}
q|_{\Gamma_1 \cap \Gamma_2} = 1 \text{ and }
q|_{B} =0
\end{equation}
and is to be determined only for the states in $w$. 
If a state $i \in w$ cannot communicate with
 $\Gamma_1 \cap \Gamma_2$, $q(i)$ is trivially $0$; let $w'$ denote
the set of states in $w$ that can communicate with $\Gamma_1 \cap \Gamma_2.$
The Markov property of $\bar{X}$ implies that for $i \in w'$
\[
q(i) = \sum_{j \in w'} \bar{\lambda}(i,j) q(j) + 
\sum_{j \in (\Gamma_1\cap \Gamma_2)} \bar{\lambda}(i,j);
\]
or in matrix notation (see \eqref{e:submatrix}):
\begin{align*}
q|_{w'} &= 
\left (\bar{\lambda}|_{w'}\right) q|_{w'} + 
\left(\bar{\lambda}|_{w' \times (\Gamma_1 \cap \Gamma_2)}\right) 
 {\bf 1}|_{\Gamma_1 \cap \Gamma_2}\\
(I - \bar{\lambda})|_{w' } ~~q|_{w'} &= 
\left(\bar{\lambda}|_{w' \times (\Gamma_1 \cap \Gamma_2)}\right) 
 {\bf 1}|_{\Gamma_1 \cap \Gamma_2}
\\
( -D \lambda )|_{w' } ~~q|_{w'} &= 
\left(\bar{\lambda}|_{w' \times (\Gamma_1 \cap \Gamma_2)}\right) 
 {\bf 1}|_{\Gamma_1 \cap \Gamma_2}.
\intertext{
For $i\neq j$,
$\bar{\lambda}(i,j) = -\lambda(i,j)/\lambda(i,i) = 
(D\lambda)(i,j)$ 
and in particular the same holds for $(i,j) \in w' \times (\Gamma_1\cap \Gamma_2)$ and therefore}
( -D \lambda )|_{w' } ~~q|_{w'} &= 
\left(-D\lambda |_{w' \times (\Gamma_1 \cap \Gamma_2)}\right) 
 {\bf 1}|_{\Gamma_1 \cap \Gamma_2}.
\end{align*}
There is no harm in taking the diagonal $D$ out of the projection operation
on both sides of the last display:
\[
 \lambda |_{w' } ~~q|_{w'} = 
\lambda |_{w' \times (\Gamma_1 \cap \Gamma_2)}
 {\bf 1}|_{\Gamma_1 \cap \Gamma_2}.
\]
That all states in $w'$ can communicate with $\Gamma_1 \cap \Gamma_2$
implies that the matrix on the left is invertible and therefore
\begin{equation}\label{e:defqw'}
q|_{w'} = ( \lambda |_{w'} )^{-1} 
\lambda |_{w' \times (\Gamma_1 \cap \Gamma_2)}
 {\bf 1}|_{\Gamma_1 \cap \Gamma_2}.
\end{equation}

\subsection{A second representation of tail probabilities}\label{ss:secreptail}
For a nonnegative integer $n$, denote by ${\mathcal P}(n)$ 
the set of all {\em sub}permutations of $\{1,2,3$ $,...,n\}$, e.g.,
${\mathcal P}(2) = \{ \emptyset, (1), (2), (1,2),(2,1) \}.$
The tail probability formula
\eqref{e:recfortailprobs} conditions on the first time $\tau'$ that
 one of the equality constraints is attained in the time interval 
$[0,t_1]$ and writes what happens after that as a recursion. 
What can happen between $\tau'$ and $t_1$?
A number of other equalities can be attained and
rather than pushing these into the recursion, one can treat
them inside the integral using a density similar to 
\eqref{e:densityt1tk}:
\begin{align}\label{e:recfortailprobsalt}
{\bf p}&(s_1,s_2,t) = \notag\\
&~~\sum_{ \pi \in {\mathcal P}(|s_2|) } \left 
( \int_{A_\pi} \left( \prod_{n=1}^{|\pi|} e^{(v_n-v_{n-1})\bm{\lambda}(W_n)} J_n
\right) e^{ (t_1- v_{|\pi|}) \bm{\lambda}(W)} dv \right) \\
&~~~~~\cdot {\bf p}(L s_1 , s_2 - s_2(\pi) + s_1(1), Lt -t_1 ),\notag
\end{align}
where
$v_0 = 0$ and
\begin{align*}
W_n 
&\doteq [S(s_1 \cup s_2 - \cup_{n_1=1}^{n}s_2(\pi(n_1)))]^c,~~
T_n \doteq \left[\bigcap_{k \in s_2(\pi(n))} \Gamma_k\right] \cap W_{n+1},\\
s_2(\pi) & \doteq \cup_{m=1}^{|\pi|} s_2(\pi(m)),~~ W \doteq [S(s_1 \cup s_2 - s_2(\pi))]^c,\\
A_\pi &\doteq 
\left\{ 
v \in {\mathbb R}^{|\pi|}: 0 < v_1 <v_2 < \dots < v_{|\pi|} \le t_1  
\right\},\\
J_n &\doteq \bm{\lambda}(W_n,T_n),
\end{align*}
$dv$ is the $|\pi|$ dimensional Lebesgue measure on 
${\mathbb R}^{|\pi|}$ 
for $|\pi| > 0$; 
$A_\pi \doteq \{0\}$ 
and $dv$ is the trivial measure on $\{0\}$
for $|\pi| = 0$.
The proof involves no additional ideas and is omitted.

\subsection{Conditional formulas}\label{ss:conditional}
The proof of Theorem \ref{t:densityt1tk} shows how one can use the density formula \eqref{e:densityt1tk}
to write down the regular conditional distribution of $\tau$ given ${\mathscr F}_{\tau'}$.
One can do the same for 
${\mathscr F}_{u_0}$, where $u_0\in {\mathbb R}_+$ is a given deterministic time.
To that end, introduce the set valued process
\[
V_u \doteq \{ k \in K, \tau_k < u\}.
\]
$K$ is finite, then so is its power set $2^K$, thus $V_u$ takes
values in a finite set.
$V_u$ is the collection of $\Gamma_k$ that $X$ has visited up to time $u$.
For ease of notation we will denote the complement of $V_u$ by $\bar{V}_u$.
The times $\tau|_{V_{u_0}}$ are known by time $u_0$ and hence 
they are constant
given ${\mathscr F}_{u_0}$. 
Thus, we only need to write down the regular conditional density
of $\tau|_{\bar{V}_{u_0}}$, i.e., the hitting times to the $\Gamma_k$ that have not been visited
by time $u_0$. From here on the idea is the same as in the proof of Theorem \ref{t:densityt1tk}.
Define $\hat{X}_u \doteq X_{u + u_0}$ and for $k \in \bar{V}_{u_0}$
\[
\hat{\tau}_k \doteq \inf \{u : \hat{X}_u \in \Gamma_k \}.
\]
The definitions of $\hat{X}$ and $\hat{\tau}$ imply
\begin{equation}\label{e:rectau}
\hat{\tau} = \tau|_{\bar{V}_{u_0}} - u_0.
\end{equation}
$\hat{X}_0 = X_{u_0}$ is a constant given ${\mathscr F}_{u_0}$.
Thus the process $\hat{X}$ has exactly the same distribution as $X$ with initial point
$X_{u_0}$ and Theorem \ref{t:densityt1tk} applies and gives the density of $\hat{\tau}$,
which is, by \eqref{e:rectau}, the regular conditional distribution of $\tau|_{\bar{V}_{u_0}}-u_0$.
Therefore, for any bounded 
measurable $g:{\mathbb R}^{\bar{V}_{u_0}} \rightarrow {\mathbb R}$ 
and a partition $s'$ of $\bar{V}_{u_0}$
\[
{\mathbb E}\left[g
\left( 
\tau|_{\bar{V}_{u_0}} \right)
1_{R_{s'}}
\left( 
\tau|_{\bar{V}_{u_0}}
\right)
  | {\mathscr F}_{u_0} \right]= \int_{R_{s'}}
g(u_0 + u) f(\delta_{X_{u_0}}, u, \bar{V}_{u_0}) d_{s'}u.
\]
We record this as
\begin{proposition}\label{p:conditional}
The regular conditional density of $\tau|_{\bar{V}_{u_0}}-t_0$
given ${\mathscr F}_{u_0}$ is 
 $f(\delta_{X_{u_0}}, t, \bar{V}_{u_0})$.
\end{proposition}

\section{Absorbing $\{\Gamma_k\}$ and connections to earlier results}\label{s:absorbing}
A nonempty subset $a \subset \Omega_0$ is said to be {\em absorbing} if
$\lambda(i,j) = 0$ for all $ i \in a$ and $j \in a^c$, i.e.,
if $\bm{\lambda}(a,a^c) = 0$.
We next derive an alternative expression
for the density formula \eqref{e:densityt1tk} under the
assumption that
all $\{ \Gamma_k, k \in K \}$ are absorbing.
The first step is 
\begin{proposition}\label{p:eltwhenabsorbing}
\begin{equation}\label{e:eltwhenabsorbing}
p_{\alpha,a}^u = \alpha e^{\bm{\lambda}(a^c) u} = 
\alpha e^{ \lambda u} I_{a^c}
\end{equation}
if $a$ is absorbing and $\alpha|_{a} = 0.$
\end{proposition}
\begin{proof}
We already know from Lemma \ref{l:avoidG1} that the first equality holds.
Therefore, it only remains to show 
$p_{\alpha,a}^u = \alpha e^{\lambda u} I_{a^c}.$
\cite[Theorem 3.4, page 48]{asmussen2003applied} 
implies that the distribution of $X$ at time
$u$ is $\alpha e^{\lambda u}$, i.e., 
$P_\alpha(X_u = j ) =  [\alpha e^{\lambda u}](j)$
for all $j \in \Omega_0$. 
That $a$ is absorbing implies that
if $X_{u_0}\in a$ then $X_u \in a$ 
for all $u \ge u_0$,
Therefore for $j \in a^c$
\[
P_\alpha( X_u= j)  = P_\alpha( 
X_u = j, X_v(\omega) \notin a , v\le u ),
\]
i.e.,
\begin{equation}\label{e:eltwhenabsorbinge1}
 (\alpha p_{\alpha,a}^u)|_{a^c} = (\alpha e^{\lambda u} I_{a^c})|_{a^c}.
\end{equation}
The definition of $p_{\alpha,a}^u$
and $\alpha|_{a} = 0$ imply
 $(\alpha p_{\alpha,a}^u)|_{a} = 0$; The definition of $I_{a^c}$ implies
$(\alpha e^{\lambda u} I_{a^c})|_{a} = 0$. This and \eqref{e:eltwhenabsorbinge1}
imply \eqref{e:eltwhenabsorbing}.
\end{proof}
Proposition \ref{p:eltwhenabsorbing} says the following:
if $a$ is absorbing 
then
$\alpha e^{\bm{\lambda}(a^c) u}$ 
is the same as $\alpha e^{ \lambda u} I_{a^c}$ 
and both describe the probability of each state in
$a^c$ at time $t$ over all paths 
that avoid $a$ in the time interval $[0,t]$.
The first expression
ensures that all paths under consideration avoid the set 
$a$ by setting the jump rates into $a$ to $0$.
The second expression does this by striking out those paths 
that end
up in one of the states in $a$ 
(the $I_{a^c}$ term does this); 
this is enough because $a$ is absorbing: once a path gets
into $a$ it will stay there and
one can look at the path's position at time $t$ to figure out
whether its weight should contribute to $p_{\alpha,a}^u$.
In the general case this is not possible, hence the
$\bm{\lambda}(a^c)$ in the exponent.

The previous proposition implies that one can replace the
$\bm{\lambda}(W_n)$ 
in the density formula
\eqref{e:densityt1tk} with $\lambda$:
\begin{proposition}\label{p:densforabs}
For $s \in {\mathcal S}$  and $t\in R_s$
let $\bar{t}$ be defined as in \eqref{e:bartWT} and 
let $f$ be the density given in Theorem \ref{t:densityt1tk}.
Then
\begin{equation}\label{e:densforabs}
f(\alpha, t,K)=
 \alpha \left( \prod_{n=1}^{|s|} e^{\lambda(\bar{t}_{n} -\bar{t}_{n-1})}  
\bm{\lambda}(W_n,T_n) \right) {\bf 1}
\end{equation}
if all $\Gamma_k$ are absorbing.
\end{proposition}
\begin{proof}
Set $\hat{\alpha}_0 = \alpha_0 \doteq \alpha$, 
\begin{equation*}
\hat{\beta}_n \doteq \hat{\alpha}_n e^{\lambda(\bar{t}_n -\bar{t}_{n-1})},~~
\beta_n \doteq \alpha_n e^{\bm{\lambda}(W_n)(\bar{t}_n -\bar{t}_{n-1})},
\end{equation*}
$n \in \{0,1,2,3,...,|s|\}$
and for $n > 0$
\begin{equation*}
\hat{\alpha}_n \doteq \hat{\beta}_{n-1}  
\bm{\lambda}(W_n,T_n),~~~
\alpha_n \doteq  \beta_{n-1} 
\bm{\lambda}(W_n,T_n).
\end{equation*}
The right side of \eqref{e:densforabs} is $\hat{\alpha}_{|s|} ~{\bf 1}$
and its left side is $\alpha_{|s|}~ {\bf 1}.$
We will prove
\begin{equation}\label{e:toprovedensabs}
\alpha_n = \hat{\alpha}_n
\end{equation}
by induction; setting $n=|s|$ in the last display
will give \eqref{e:densforabs}. 
For $n=0$ \eqref{e:toprovedensabs}
is true by definition; 
assume that it holds for $0 <n-1 < |s|$; we will argue that this
implies that it must also for $n$.
Union of absorbing sets is also absorbing, therefore
$S(L^{n-1}s)$ is absorbing. 
This, $W_n = \Omega_0 - S(L^{n-1}s)$,
the induction hypothesis and \eqref{e:eltwhenabsorbing} (set $a = S(L^{n-1}s)$)
imply
\begin{align}
\alpha_n = \beta_n \bm{\lambda}(W_n,T_n)   &= \notag 
\alpha_{n-1} e^{\bm{\lambda}(W_n)(\bar{t}_n -\bar{t}_{n-1})} 
\bm{\lambda}(W_n,T_n)  \\
&=\notag
\hat{\alpha}_{n-1} 
 e^{\lambda(\bar{t}_n -\bar{t}_{n-1})} I_{W_n} \bm{\lambda}(W_n,T_n)\\
&= \hat{\beta}_{n-1}  I_{W_n} \bm{\lambda}(W_n,T_n).\notag
\intertext{The identities \eqref{e:lambdamatrix} and \eqref{e:commutation} imply
$I_{ W_n} \bm{\lambda}(W_n,T_n) = \bm{\lambda}(W_n,T_n)$ and therefore}
&= \hat{\beta}_{n-1} \bm{\lambda}(W_n,T_n) =\hat{\alpha}_n.
\notag 
\end{align}
This completes the induction step and therefore the whole proof.
\end{proof}

Using the same ideas and calculations
as in the previous proof
one can write the tail probability
formula \eqref{e:recfortailprobs}
as 
\begin{align*}
&{\bf p}(s_1,s_2,t) = \\
&~~
\int_0^{t_1} e^{\lambda u} 
\bm{\lambda}(W, T(s_1,s_2))
\left( \sum_{n=1}^{|s_2|} I_{T_n(s_1,s_2)}~~ {\bf p}(s_1, s_2 - s_2(n), t-u)\right) du
\notag \\
&~~~+e^{\lambda t_1  }I_{W}
 {\bf p}\left( L s_1 , s_2 + s_1(1), Lt - t_1
\right)\notag
\end{align*}
and \eqref{e:tailsc}  as
\begin{equation}\label{e:tailscabs}
{\bf p}(s_1,\emptyset,t) = e^{\lambda t_1}I_{S(s_1)^c} 
{\bf p}(s_1 - s_1(1),
( s_1(1) ), Lt - t_1 )
\end{equation}
when $\{\Gamma_k, k \in K \}$ are absorbing.

Let us briefly point out another possible modification of the density
formula for absorbing $\{\Gamma_k\}.$ Define
\[
\hat{T}_0 = \Omega_0 - S(s),~~~
\hat{T}'_n \doteq \bigcap_{ \cup_{ m \le n } s(m)} \Gamma_k,~~~
\hat{T}_n = \hat{T}'_n 
- S(L^n(s)),~~~ \hat{W}_{n} = \hat{T}_{n-1},
\]
where $s \in {\mathcal S}$ and $n \in \{1,2,3,...,|s|\}$.
If $\{\Gamma_k\}$ are absorbing one can replace the target
and waiting sets $T_n$ and $W_n$  of  \eqref{e:bartWT}  with
$\hat{T}_n$ and $\hat{W}_n$ defined above. One can prove that
the density formula continues to hold after this modification
with an argument
parallel to the proof of Proposition \ref{p:densforabs} 
using in addition that the intersection
of absorbing sets is again absorbing.

\subsection{Connections to earlier results}\label{ss:linkstoearlier}\label{ss:linkstoearlier}
This subsection gives several examples of how to express
density/distribution formulas 
from the prior
phase-type distributions literature
as special cases of the ones derived in 
the present work.

We begin by three formulas from \cite{assaf1984multivariate}. The first two concern
a single hitting time and the last a pair. \cite{assaf1984multivariate} denotes
the state space of $X$ by $E$ 
assumes that it has an absorbing element $\Delta$,
denotes 
$\lambda|_{{ \left\{ \Delta \right \}^c }}$ 
by $A$ and $\inf\{u: X_u = \Delta \}$
by $T$. 
\cite{assaf1984multivariate} 
also uses the letter 
$\alpha$ 
to denote the initial distribution of $X$, 
but over the set $\hat{E} \doteq E - \{ \Delta \}$ 
(and implicitly assuming $P(X_0 = \Delta) = 0$).
We will use the symbol $\hat{\alpha}$
to denote the `$\alpha$ of \cite{assaf1984multivariate}.'
 The relation between
$\alpha$ and $\hat{\alpha}$ is
$\alpha|_{\{\Delta\}^c} = \hat{\alpha}$.

The first line of
\cite[equation (2), page 690]{assaf1984multivariate} 
says
$P_\alpha(T > u) = \hat{\alpha} e^{Au} {\bf e}$ where ${\bf e}$ is the $|E|-1$
dimensional vector with all component equal to $1$. The corresponding formula
in the present work is \eqref{e:tail1} where one takes $d = \{\Delta\}$.
The following facts imply the equality of these formulas
1) 
$\bm{\lambda}(d^c)|_{d^c} = A$ 
and 2) 
the row of $\bm{\lambda}(d^c)$ corresponding to $\Delta$ is $0$.
The second line of the same equation gives $-\hat{\alpha} e^{uA} A{\bf e}$ 
as the
density of $T$. The corresponding formula here is \eqref{e:dist0} with $b=\emptyset$,
and $a=\{ \Delta \}$
for which it reduces to $e^{u\bm{\lambda}(a^c) } \bm{\lambda}(a^c,a) {\bf 1}$.
This time, 1), 2) and the following fact imply the equality of the formulas:
the row sums of $\lambda$ are zero, therefore
$\bm{\lambda}(a^c, \Delta)|_{a^c} = \lambda|_{a^c} {\bf e} = A { \bf e}.$
The matrix $\bm{\lambda}(a^c, a)$ is  the column of $\lambda$ corresponding to
$\Delta$; one way to write it is as the negative of
 the sums of the rest of the columns, this is
what the last equality says.

\cite[Equation (5), page 692]{assaf1984multivariate} concerns the following setup 
(using the notation of that paper):
we are given two set $\Gamma_1, \Gamma_2 \subset E$ with $\Gamma_1 \cap \Gamma_2 = 
\{ \Delta\}$, $T_i$ is the first hitting time to $\Gamma_k.$ 
The formula just cited
says
\begin{equation}\label{e:asad2d}
P_\alpha( T_1 = T_2 > u) = \hat{\alpha} e^{Au} A^{-1} 
( A g_1 g_2 - [A,g_1]- [A,g_2] ) {\bf e},
\end{equation}
where $g_i = I_{\Gamma_k}|_{{\{\Delta\}^c}}$ 
and for two matrices $B$ and $C$, $[B,C] \doteq BC- CB$.
The absorbing property of $\Gamma_1$ and $\Gamma_2$ implies that
the matrix inside the parenthesis in the last display equals
$g'A$, 
where $g'=I_{(\Gamma_1 \cup \Gamma_2)^c}|_{\hat{E}}$
i.e., the same matrix as $A$ except that the rows whose indices
appear in $\Gamma_1 \cup \Gamma_2$ are replaced with $0$. Thus
$( A g_1 g_2 - [A,g_1]- [A,g_2] ) {\bf e}$ is another way to
take the $\Delta$ column of $\lambda$ and replace its components
whose indices appear in $\Gamma_1 \cup \Gamma_2$ with $0$. 
Denote
this vector by $C_\Delta$. Then the right side of
\eqref{e:asad2d} is 
\begin{equation}\label{e:abbcomp}
\alpha|_{\hat{E}} \left( e^{\lambda u}|_{\hat{E}}\right)
 A^{-1} C_\Delta.
\end{equation}
The same probability is expressed by a special case of \eqref{e:tailscabs};
for the present case one sets $K=\{1,2\}$, $s_1 = (\{1,2\} )$; for these
values, \eqref{e:tailsc} and
conditioning on the initial state gives
\begin{equation}\label{e:abbascur}
P_\alpha( \tau_1 = \tau_2 > u) = 
\alpha e^{\lambda u} 
I_{w}
{\bf p}(\emptyset, (\{1,2\}), 0),
\end{equation}
where
$w=(\Gamma_1 \cup \Gamma_2)^c$. 
Remember that
we have denoted the last probability as $q$
and derived for it the formulas \eqref{e:bcq}  and \eqref{e:defqw'}.
The article
\cite{assaf1984multivariate} assumes that all states can communicate
with $\Delta$, which implies that $w'$ of \eqref{e:defqw'} equals $w$.
This 
and $\Gamma_1 \cap \Gamma_2 = \{\Delta\}$ imply
 $ \lambda|_{w'\times \Gamma_1 \cap \Gamma_2} {\bf 1}$ in \eqref{e:defqw'}
equals $\lambda|_{(\Gamma_1 \cup \Gamma_2)^c \times \Delta}$
i.e.,  the $\Delta$ column of $\lambda$ projected to its 
indices in $(\Gamma_1 \cup\Gamma_2)^c$, i.e., 
$C_\Delta|_{w}$.   The only difference between $C_\Delta$ and
$C_\Delta|_{w}$ is that the former has zeros in its extra dimensions.
This and the absorbing property of $\Gamma_k$ imply
\[(\lambda|_w)^{-1} C_\Delta|_w = (A^{-1} C_\Delta)|_w.
\]
Note that we are commuting the projection operation and the inverse
operation; this is where the absorbing property is needed.
The last display, \eqref{e:bcq} and \eqref{e:defqw'} give
$I_{w}
{\bf p}(\emptyset, (\{1,2\}), 0)|_{\hat{E}}= q|_{\hat{E}}
=  A^{-1} C_\Delta$.
This and $\alpha(\{\Delta\}) =0$ imply that
one can rewrite the right side of \eqref{e:abbascur} as
\[
\alpha e^{\lambda u } ( A^{-1} C_\Delta)|^{E}.
\]
Once again $\alpha(\Delta)=0$ 
implies that the last expression 
equals \eqref{e:abbcomp}.

The density formula \cite[(3.1.11)]{herbertsson2011modelling} will provide our last example.
The process $X$ of \cite{herbertsson2011modelling} 
is a random walk (with absorbing boundary) 
on ${\mathbb Z}_2^m$ 
with
increments $\{-e_k, k=1,2,3,...,m\}$ where $e_k$ is the unit vector with $k^{th}$ coordinate
equal to $1$ (\cite{herbertsson2011modelling} uses different but equivalent notation, in particular
the name of the process is $Y$ and its state space is represented by subsets of $\{1,2,3,...,m\}$; the notation
of this paragraph is chosen to ease discussion here and in the ensuing sections).
\cite{herbertsson2011modelling} takes
$\Gamma_k = \{ z \in {\mathbb Z}_2^m: z_k = 0 \}$ ($\Delta_n$, see the display after \cite[(2.3)]{herbertsson2011modelling}).
and assumes them to be absorbing.
The jump rate for the increment $-e_k$ is assumed to be
$\langle X, b_k \rangle + a_k$ for fixed
$b_k \in {\mathbb R}^m$ and $a_k \in {\mathbb R}$ (given in \cite[(2.1)]{herbertsson2011modelling}).
A key property of this setup is this: take any 
collection $\{\Gamma_{k_1}, \Gamma_{k_2},...,\Gamma_{k_n} \}$
with $n>1$; because the only increments of $X$ are the $\{-e_k\}$,
the process cannot enter the sets in the collection at the same time.
Thus, in this formulation,
 $X$ must hit the $\{\Gamma_k\}$
at separate times and 
the
distribution of $\tau$ has no singular part, i.e., $P( \tau \in R_s) =0$ for $|s| < m,$  and one needs only
the density of $\tau$ with respect to the full 
Lebesgue measure in ${\mathbb R}^m$ (the ``absolutely 
continuous part''). As noted in \cite{herbertsson2011modelling},
 this is already available in \cite{assaf1984multivariate} 
(see the display following (7)
on page 694) and is given in \cite[display (3.1.1)]{herbertsson2011modelling} as follows:
\begin{equation}\label{e:assafac}
f(t) = (-1)^m \alpha \left( \prod_{n=1}^{m-1} e^{\lambda (\bar{t}_n - \bar{t}_{n-1}) } ( \lambda G_{k_n} - G_{k_n} \lambda) \right) 
e^{ \lambda(\bar{t}_m - \bar{t}_{m-1})} \lambda G_{k_m} {\bf 1},
\end{equation}
for $t \in R_s$ with $|s| = m$; here
$G_k = I_{\Gamma_k^c}$ and $k_n$ is the index for which $t_{k_n}  = \bar{t}_n$ (\cite{herbertsson2011modelling}
uses the letter $Q$ for the rate matrix $\lambda$).
We briefly indicate why \eqref{e:densforabs}  is equivalent to the last formula with
the assumptions of this paragraph, i.e., when
the dynamics of $X$
precludes it to enter more than one of the $\{\Gamma_k\}$ at the same time and in particular when
$|s|$ equals the dimension of $\tau$ (denoted by $m$ in the current paragraph).
Lemma \ref{l:Iaaction} and the
absorbing property of $\Gamma_k$ imply
\begin{align*}
\lambda G_{k} - G_{k} \lambda &= \bm{\lambda}(\Gamma_k,\Gamma_k^c) - \bm{\lambda}(\Gamma_k^c,\Gamma_k)\\
 &=  - \bm{\lambda}(\Gamma_k^c,\Gamma_k).
\end{align*}
On the other hand again Lemma \ref{l:Iaaction} and the absorbing property of $\Gamma_{k}$
imply  $\lambda G_{k} = -\bm{\lambda}(\Gamma_{k}^c ,\Gamma_{k}^c )$.
The row sums of $\lambda$ equal $0$. 
The last two facts imply $\lambda G_{k} {\bf 1} = -\bm{\lambda}(\Gamma_{k}^c, \Gamma_{k}){\bf 1}.$
These imply that one can write \eqref{e:assafac} as 
\[
f(t) =  \alpha \left( \prod_{n=1}^{m} e^{\lambda (\bar{t}_n - \bar{t}_{n-1}) } \bm{\lambda}(\Gamma_{k_n}^c,\Gamma_{k_n})\right)
{\bf 1}.
\]
As we noted above, for $k\neq k'$ the dynamics of $X$ imply that it cannot enter $\Gamma_k$ and $\Gamma_{k'}$ 
at the same time.
Furthermore,
by definition $t_n \neq t_{n'}$  for $n \neq n'$. 
Finally, the initial distribution $\alpha$ is assumed to be such that it puts
zero mass on $\cup_{k\in K}^m \Gamma_k.$ These imply
that
one can replace $\bm{\lambda}(W_n,T_n)$ of \eqref{e:densforabs} with $\bm{\lambda}(\Gamma_{k_n}^c,\Gamma_{k_n})$
(a full argument requires an induction similar to the proof of Proposition \ref{p:densforabs}),
and therefore under the current assumptions the last display and \eqref{e:densforabs} are equal.

\section{Numerical Example}\label{s:numerical}
The state space of our numerical
example is
$\Omega_0 = {\mathbb Z}_3^3$.
For $z \in {\mathbb Z}_3^3$ and $k \in  K  =\{1,2,3\}$
let $z_k$ denote the $k^{th}$ component of $z$.
For the collection $\{\Gamma_k\}$ take
\[
\Gamma_k = \{z: z_k = 0 \}.
\]
$\tau_k$, as before, is the first time the process $X$ hits the set $\Gamma_k$.
The initial distribution
$\alpha$ will be the uniform distribution over the set 
\[
\Omega_0- \bigcup_{k \in K} \Gamma_k  = \left\{z: \min_{k\in K} z_k > 0 \right\}.
\]
We will compute the density of $\tau = (\tau_1,\tau_2,\tau_3)$
over the sets $R_{s_1}$, $R_{s_2}\subset {\mathbb R}_+^3$ 
defined by the partitions
$s_1 = ( \{2,3\},\{1\})$ 
and
$ s_2= (\{1,2,3\})$;
the first corresponds to the event 
$\{ \tau \in R_{s_1} \} = \{\tau_2 < \tau_1 = \tau_3 \}$
and the second to
$\{ \tau \in R_{s_2} \} = 
\{ \tau_1 = \tau_2 = \tau_3\}$.

The dynamics of $X$ on ${\mathbb Z}^3_3$ 
for our numerical example will be that
of a constrained random walk with the following increments:
\begin{equation}\label{e:incs}
\pm e_k, \pm(e_1 + e_2), \pm(e_1 + e_2 + e_3), k \in K,
\end{equation}
where $e_1 \doteq (1,0,0)$, $e_2 \doteq (0,1,0)$ and $e_3 \doteq (0,0,1)$;
the $\{ \Gamma_k \}$ are assumed to be absorbing, i.e., if $X_{u_0} \in \Gamma_k$
any increment involving $\pm e_k$ 
can no longer be an increment of $X$ for $u > u_0$.
The sets $B_k \doteq \{z: z_k = 2\}$ are ``reflecting'' in the sense that 
if  $X_t \in B_k$ for some $t$,
increments involving $+e_k$ cannot be the first 
increment of $X$ in the time interval $[t,\infty)$. 
We assume the following
jump rates for the increments listed in \eqref{e:incs}:
\[
2~, 1~, 2~, 1~, 3~, 1~, 0.5~, 0.5~, 0.2~, 0.2.
\]
These rates and the aforementioned dynamics
give a $27 \times 27$ $\lambda$ matrix. The level sets 
$f(\alpha,\cdot ,K)|_{R_{s_1}}$ are depicted in Figure \ref{f:t2e31} and
the graph of $f(\alpha,\cdot,K)|_{R_{s_2}}$ is depicted in Figure \ref{f:t1e2e3}.

\begin{figure}
\begin{center}
\scalebox{0.5}{
\centerline{\epsfig{file=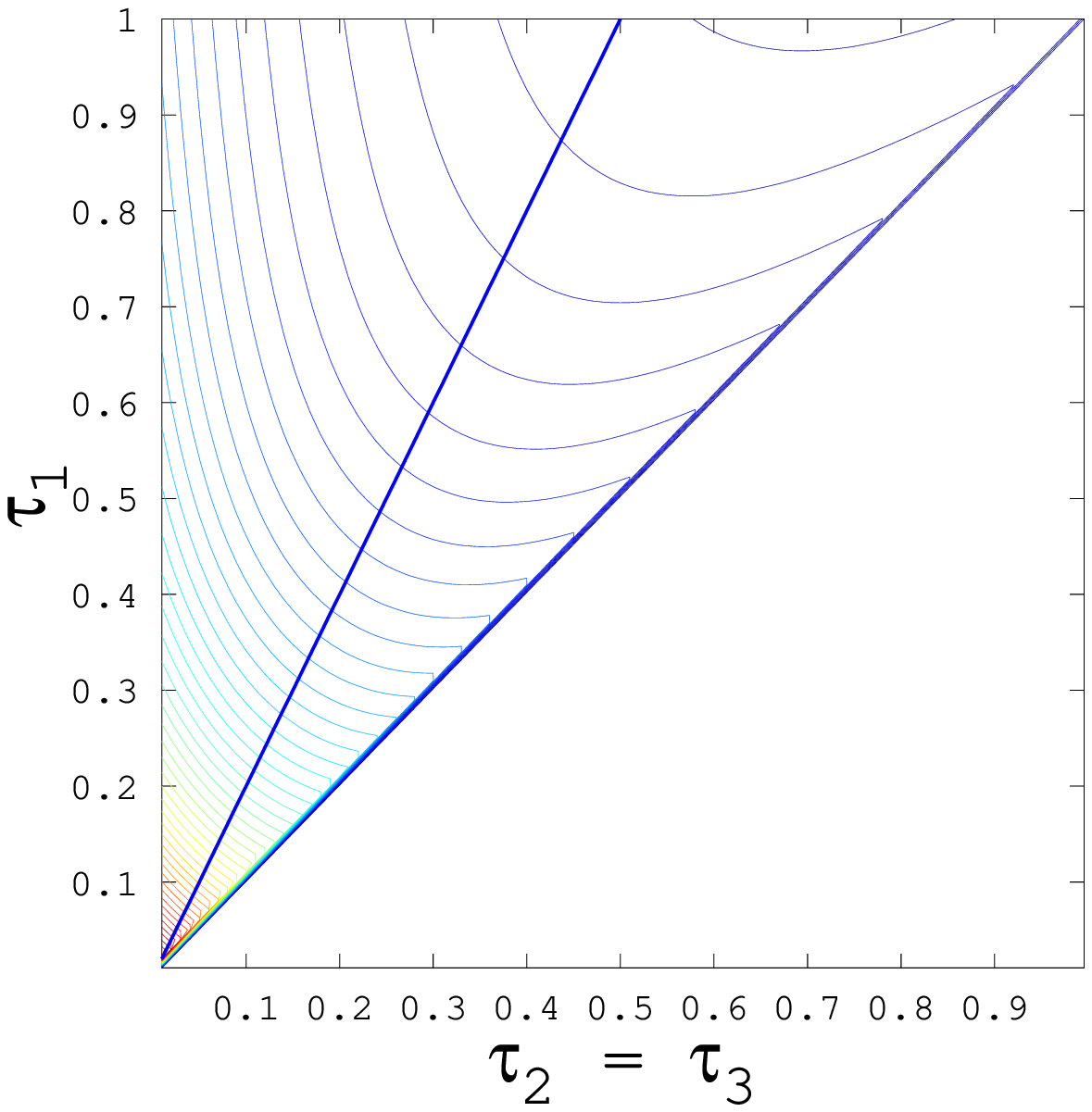} \epsfig{file=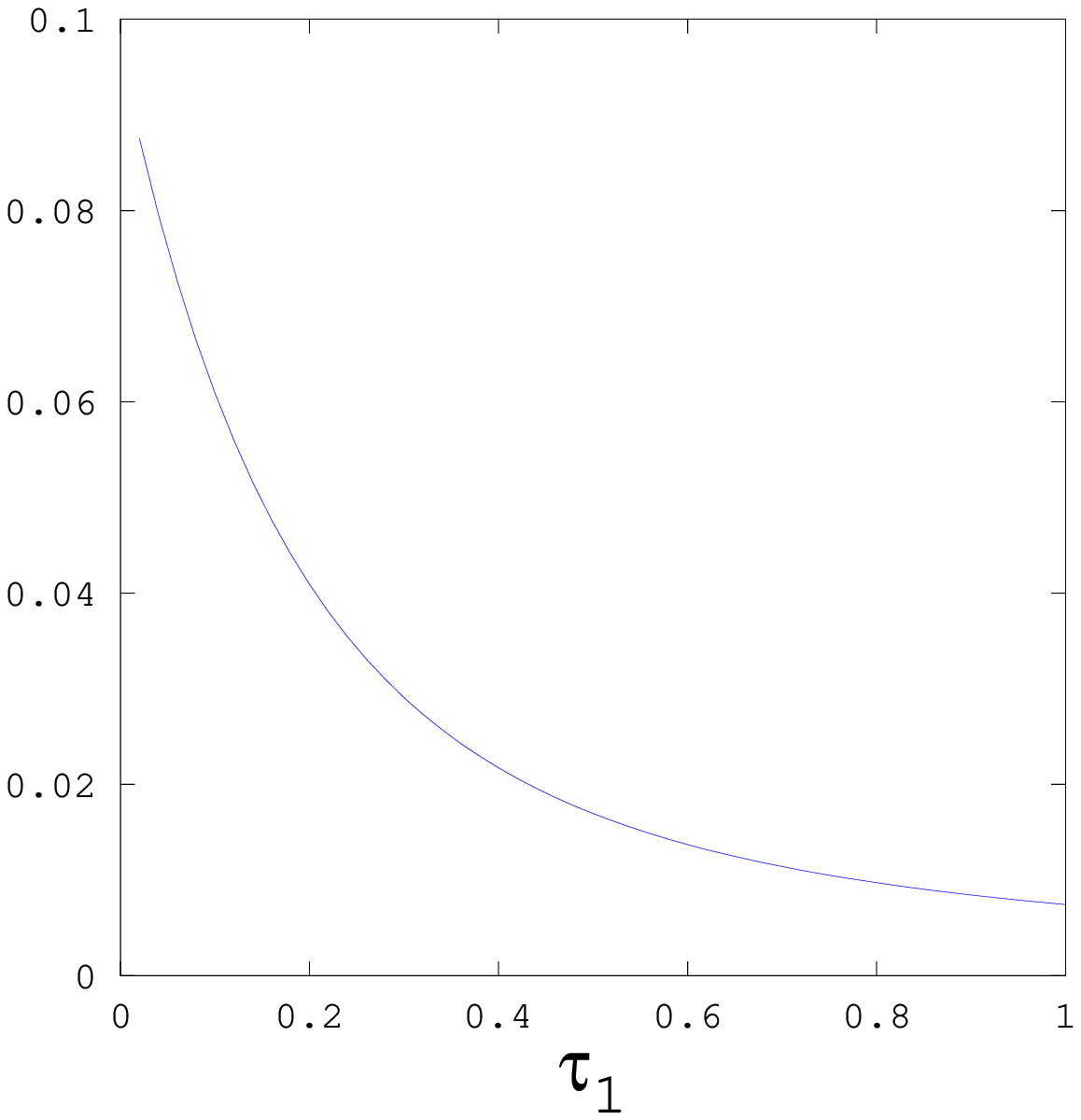}}}
\end{center}

\vspace{-0.8cm}
\caption{The level curves of the density 
$f$ for $\tau_2 = \tau_3  < \tau_1$.
On the right: the values of $f$ over
the line segment connecting $(0,0)$ to  $(0.5,1)$}
\label{f:t2e31}
\end{figure}

\begin{figure}
\begin{center}
\scalebox{0.5}{
\centerline{\epsfig{file=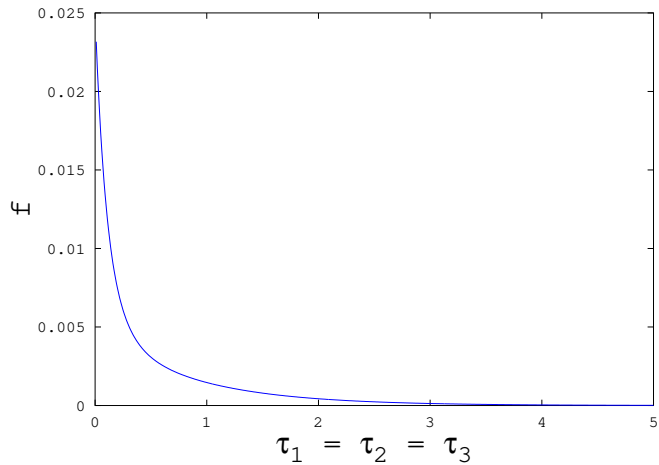} }
}
\end{center}

\vspace{-0.8cm}
\caption{The density $f$ for $\tau_1 = \tau_2 = \tau_3$}
\label{f:t1e2e3}
\end{figure}

For the parameter values of this numerical example, 
$P_\alpha( 
\cap_{k\neq k'} 
\tau_{k} \neq \tau_{k'} 
) = 0.899$ and thus the singular parts 
account for around
$10\%$ of the distribution of $\tau$.

\section{Conclusion}
Our primary motivation in deriving the formulas in the present paper has been
their potential applications to credit risk modeling. Let us
comment on this potentiality starting from the credit risk model of \cite{herbertsson2011modelling}.
With the results in the present work one can extend the modeling approach of
\cite{herbertsson2011modelling} in two directions. Remember that the underlying process in
\cite{herbertsson2011modelling} can only move by increments of $\{-e_k\}$ i.e., the model
assumes that the obligors can default only one at a time. 
However, for highly correlated
obligors it may make sense to allow
simultaneous defaults, i.e., allow increments of the form $-\sum_{n} e_{k_n}$.
Once multiple defaults are allowed the default times will have nonzero singular parts and
the formulas in the present work can be used to compute them, as is done in the numerical example
of Section \ref{s:numerical}. Secondly, the default sets $\{\Gamma_k\}$ no longer have to be assumed
to be absorbing. Thus, with our formulas, one can treat models that allow recovery from default.

As $|\Omega_0|$ increases 
\eqref{e:densityt1tk} and other formulas derived in the present paper can take too long a time to compute
(the same holds for earlier density formulas in the prior literature).
Thus it is of  interest to derive asymptotic
approximations for these densities.

\appendix
\bibliography{phasetype}

\end{document}